\documentclass[a4paper,12pt]{article}
\usepackage[cp1251]{inputenc}
\usepackage[ukrainian, russian, english]{babel}

\usepackage{graphicx}
\usepackage{epstopdf}
\usepackage{amsmath,amssymb,euscript,amsthm}
\usepackage{placeins}
\newcommand{\GAM}{\Gamma\!\!\!\Gamma}
\newcommand{\PHI}{\Phi\!\!\!\Phi}

\theoremstyle{plain}
\newtheorem{thm}{Theorem}

\theoremstyle{definition}
\newtheorem{nsl}{Corollary}
\theoremstyle{remark}

\newtheorem{prop}{Proposition}

\begin{document}

\begin{center}

{\vspace{5mm}\fontsize{20pt}{12pt}\selectfont\textbf{Filtering of stationary \\ Gaussian statistical experiments}} 
\vskip12pt
\textbf{V.S.Koroliuk, D.Koroliouk}
\end{center}

\vskip18pt
\par\textbf{Abstract.}
This article proposes a new filtering model for stationary Gaussian Markov statistical experiments, given by diffusion-type difference stochastic equations.


\section{Stationary statistical experiments}

The \textit{statistical experiment (SE)} is defined as the averaged sums:
\begin{equation}
\label{eq1}
S_N(k)=\frac{1}{N}\sum^N_{r=1}\delta_r(k), \ k\geq0,
\end{equation}
in which the random variables $\delta_r(k)$, $1\leq r\leq N$, $k\geq0$, are equally distributed and independent for each fixed $k\geq 0$, which take binary values 0 or 1.

In particular, let us consider
\begin{equation}
\nonumber
\delta_r(k)=I(A) =
\begin{cases} 1, & \text{if event $A$ occurs;} \\
              0, & \text{if event $A$ does not occur.}
\end{cases} \ , \ \ 1\leq r\leq N \ , \ \  k\geq0.
\end{equation}
In this case, the random amount $S_N(k)$, $k\geq0$, describes the relative frequencies of presence of the attribute $A$ in a sample of fixed volume $N$ at each time instant $k\geq0$.

Introducing the normalized fluctuations $\zeta_N(k):=\sqrt{N}(S_N(k)-\rho)$, where $\rho$ be the equilibrium of SE \cite{DK_6, DK_05}, one gets an important representation of SEs.

Namely, some natural conditions \cite[Proposition 5.2]{DK_8}, the SE \eqref{eq1} has the following diffusion approximation:
\begin{equation}
\label{eq2}
\Delta\zeta(k+1)=-V\zeta(k)+\sigma\Delta W(k+1) \ , \ \ 0\leq V\leq 2 \ , \ \ \sigma\geq0,
\end{equation}
where the increments $\Delta\zeta(k+1):=\zeta(k+1)-\zeta_N(k)$ and $\Delta W(k+1)$ are the standard normally distributed martingale - differences.
The solution of the difference stochastic equation \eqref{eq2} is called \textit{discrete Markov diffusion} (DMD) \cite{DK_19}.
\vskip10pt
The next theorem \cite{DK_13} gives the necessary and sufficient conditions of the stationarity, in wide sense, of the DMD \eqref{eq2}.
\begin{thm}
(Theorem on stationarity).
The DMD \eqref{eq2} is a stationary random sequence in wide sense if and only if the following relations take place:
\begin{equation}
E\zeta(0)=0 \ , \ \ E\zeta^2(0)=R_\zeta=\sigma^2/(2V-V^2).
\label{eq3}
\end{equation}
\end{thm}
\vskip10pt

Now consider a stationary, in wide sense, two-component random sequence\\ $\Bigl(\zeta(k),\ \Delta\zeta(k+1)\Bigr)$, $k\geq0$, with the following joint covariances:
\begin{equation}
\label{eq4}
\begin{split}
&R_\zeta=E\bigl[\zeta(k)\bigr]^2 \ , \ \ R^0_\zeta= E\bigl[\zeta(k)\Delta\zeta(k+1)\bigr] \ , \ \ R^\Delta_\zeta=E\bigl[\Delta\zeta(k+1)\bigl]^2.
\end{split}
\end{equation}
In filtering problem of Gaussian stationary DMD, the equivalence formulated in the following theorem (see \cite{DK_13}) is essentially used.
\begin{thm}\label{thm2}
(Theorem on equivalence).
Let the two-component Gaussian Markov random sequence \ $\bigl(\zeta(k),\ \Delta\zeta(k+1)\bigr)$, $k\geq0$, with the mean value $E[\zeta(k)]=0$, $k\geq0$, and the joint covariances \eqref{eq4} that satisfy the stationarity condition
\begin{equation}\label{eq5}
\sigma^2=\bigl(2V-V^2\bigr)R_\zeta.
\end{equation}
Then the random sequence $\Bigl(\zeta(k),\ \Delta\zeta(k+1)\Bigr)$, $k\geq0$, is a solution of the stochastic difference equation \eqref{eq2}, that is a DMD.
\end{thm}
\begin{proof} By Theorem on normal correlation \cite[Th. 13.1]{LIP-SH1974}, one has:
\begin{equation}
\label{eq6}
E\bigl[\Delta\zeta(k+1)\,|\,\zeta(k)\bigr]=R_\zeta^0 R_\zeta^{-1}\zeta(k).
\end{equation}
Hence
\begin{equation}\label{eq7}
R_\zeta^0=-V R_\zeta \ \ \ \text{and} \ \ \ R_\zeta^\Delta=2VR_\zeta.
\end{equation}
Considering the martingale-differences
\begin{equation}
\label{eq8}
\Delta W(k+1)=\frac{1}{\sigma}\biggl(\Delta\zeta_N(k+1)+V\zeta(k)\biggr).
\end{equation}
Let's calculate its first two moments.
\begin{equation}
\label{eq9}
E\biggl[\Delta W(k+1)\biggr]=\frac{1}{\sigma}E\biggl[\Delta\zeta_N(k+1)+V\zeta(k)\biggr]=0,
\end{equation}
\begin{equation}
\label{eq10}
E\biggl[\Delta W(k+1)\biggr]^2=1.
\end{equation}
Now it remains to prove that, the stochastic part covariations are:
\begin{equation}
\label{eq11}
E\biggl[\Delta W(k+1)\Delta W(r+1)\biggr]=\begin{cases}
                                            1, & \mbox{if } k=r, \\
                                            0, & \mbox{otherwise}.
                                          \end{cases}
\end{equation}
Suppose for determination, that $r<k$. Using the Markov property of the sequence  $(\zeta(k)$,\ $k\geq0$,
and the relation \eqref{eq9}, one obtains:
\begin{equation}
\label{eq11}
E\biggl[\Delta W(k+1)\,|\, \zeta(r),\zeta(k)\biggr]=E\biggl[\biggl(\Delta \zeta(k+1) +V\zeta(k)\biggr)\,|\,\zeta(k)\biggr]=0.
\end{equation}
Theorem \ref{thm2} is proved.
\end{proof}

\section{Filtering of discrete Markov diffusion}

The filtering and extrapolation problem is considered by many authors (for ex., \cite{MMS_01, MMS_02}.
In our constructions of the new filter, we proceed from the following basic principle: the presence of two normally distributed random sequences implies the presence of their covariances, which contain information about the filtering.

The task is to estimate the unknown parameters of a stationary Gaussian Markov signal process $\alpha(k)$ by using the trajectories of the signal $(\alpha(k),\,\Delta\alpha(k+1))$, and a stationary Gaussian Markov filtering process $(\beta(k),\,\Delta\beta(k+1))$, $k\geq0$.

The signal with unknown parameters is determined by the next equation:
\begin{equation}
\label{eq13}
\Delta\alpha(k+1)=-V_0\alpha(k)+\sigma_0\Delta W^0(k+1) \ , \ \ k\geq0.
\end{equation}
The filtering process - by the equation:
\begin{equation}\label{eq14}
\Delta\beta(k+1)=-V\beta(k)+\sigma\Delta W(k+1) \ , \ \ k\geq0.
\end{equation}
It is known that the best estimate (in the mean square sense) of the signal $(\alpha(k),\,\Delta\alpha(k+1))$, by observing the filtering proccess $(\beta(k),\,\Delta\beta(k+1))$, coincides with the conditional expectation
\begin{equation}\label{eq15}
(\widehat{\alpha}(k),\,\Delta\widehat{\alpha}(k+1))= E\biggl[(\alpha(k),\,\Delta\alpha(k+1))
\,\biggl|\,(\beta(k),\,\Delta\beta(k+1))\biggr].
\end{equation}
The next calculation of filtering matrix $\PHI_\beta$, determined by the conditional expectation \eqref{eq15}, essentially uses Theorem 2 on equivalence and is based on Theorem on normal correlation by Liptser and Shiryaev.

\begin{thm}
\label{thm3}
The estimate \eqref{eq15} is determined by the filtering equation
\begin{equation}
\label{eq16}
\begin{pmatrix}\widehat{\alpha}(k), & \Delta\widehat{\alpha}(k+1)\end{pmatrix}= \PHI_\beta\cdot\begin{pmatrix} \beta(k) \\ \Delta\beta(k+1)\end{pmatrix},
\end{equation}
with the filtering matrix
\begin{equation}
\label{eq17}
\PHI_\beta=\begin{bmatrix} 1 & 0 \\ -V_0 & 0
\end{bmatrix} R_{\alpha\beta}R_\beta^{-1}.
\end{equation}
where
\begin{equation}
\label{eq18}
R_{\alpha\beta}:=E\bigl[\alpha(k)\beta(k)\bigr] \ , \ \
R_{\beta}:=E\bigl[\beta^2(k)\bigr].
\end{equation}
\end{thm}
\begin{nsl}
The interpolation of the signal $\alpha(k)$  by observing the filtering process $\beta(k)$ is
\begin{align*}
&\widehat{\alpha}(k)=\PHI_{11}\beta(k) \ , \ \ \PHI_{11}:= R_{\alpha\beta}R_\beta^{-1};\\
&\Delta\widehat{\alpha}(k+1)=\PHI_{21}\beta(k)=-V_0\PHI_{11}\Delta\beta(k+1).
\end{align*}
\end{nsl}

\begin{nsl}
One has the following statistical parameter estimation:
\begin{equation}
\label{eq19}
V_0\approx V_0^T=-\frac{\PHI_{21}^T}{\PHI_{11}^T}.
\end{equation}
\end{nsl}

\begin{nsl}
One has the following statistical parameter estimation:
\begin{equation}
\label{eq20}
\sigma_\alpha\approx \sigma_\alpha^T=\mathcal{E}_0^T\cdot R_\alpha^T \ , \ \ \mathcal{E}_0^T:=2V_0^T-\bigl(V_0^T\bigr)^2.
\end{equation}
\end{nsl}
\vskip10pt
\begin{prop}
Under the assumption of mutual uncorrelatedness \eqref{eq24}, the statistical estimates \eqref{eq19}\footnote{For the filtering parameter estimation see the Section 4.}
 -- \eqref{eq20} are unbiased and strongly consistent, as $T\to\infty$.
\end{prop}

\vskip10pt

\begin{proof}[Proof of Theorem \ref{thm3}]
By the Theorem on normal correlation \cite[Th. 13.1]{LIP-SH1974}, the filtering matrix introduced in \eqref{eq16} has the following form:
\begin{equation}\label{eq21}
\PHI_\beta=\mathbb{R}_{\alpha\beta}\mathbb{R}_{\beta}^{-1},
\end{equation}
where
\begin{equation}
\label{eq22}
\mathbb{R}_{\alpha\beta}:=\begin{bmatrix} R_{\alpha\beta} & R_{\alpha\beta}^0 \\
R_{\beta\alpha}^0 & R_{\alpha\beta}^\Delta \end{bmatrix} \ , \ \
\mathbb{R}_{\beta}:=\begin{bmatrix} R_{\beta} & R_{\beta}^0, \\
R_{\beta}^0 & R_{\beta}^\Delta \end{bmatrix},
\end{equation}
and the covariances are defined as follows:
\begin{equation}
\label{eq23}
\begin{split}
&R_{\alpha\beta}^0:=E(\alpha(k)\Delta\beta(k+1)),\\ &R_{\beta\alpha}^0:=E(\beta(k)\Delta\alpha(k+1)) \ , \ \ R_{\alpha\beta}^\Delta:=E(\Delta\alpha(k+1)\Delta\beta(k+1)),\\
&R_{\beta}^0:=E(\beta(k)\Delta\beta(k+1)) \ , \ \ R_{\beta}^\Delta:=E[(\Delta\beta(k+1))^2].
\end{split}\end{equation}
Under the assumption of mutual uncorrelatedness
\begin{equation}
\label{eq24}
E\bigl[\alpha(k)\cdot W(k)\bigr]=0 \ , \ \ E\bigl[\beta(k)\cdot W^0(k)\bigr]=0
 \ \ \hbox{and} \ \ E\bigl[W(k)\cdot W^0(k)\bigr]=0 \ , \ \ k\geq0,
\end{equation}
the equations \eqref{eq13} -- \eqref{eq14} imply the following representations:
\begin{equation}
\label{eq25}
\begin{split}
&R_{\alpha\beta}^0=-VR_{\alpha\beta} \ , \ \ R_{\beta\alpha}^0=-V_0R_{\alpha\beta} \ , \ \ R_{\alpha\beta}^\Delta=VV_0R_{\alpha\beta};\\
&R_{\beta}^0=-VR_\beta \ , \ \ R_{\beta}^\Delta:=2VR_\beta.
\end{split}\end{equation}
So the matrix
\begin{equation}
\label{eq26}
\mathbb{R}_{\beta}=\begin{bmatrix} R_{\beta} & -VR_{\beta} \\
-VR_{\beta} & 2VR_{\beta} \end{bmatrix}
\end{equation}
has the following inversion:
\begin{equation}
\label{eq27}
\mathbb{R}_{\beta}^{-1}=\begin{bmatrix} 2VR_{\beta} & VR_{\beta} \\
VR_{\beta} & R_{\beta} \end{bmatrix}\cdot d_\beta^{-1} \ , \ \ d_\beta:=(2V-V^2)R_\beta^2.
\end{equation}
Now let's calculate the elements of the filtering matrix
\begin{equation}
\label{eq28}
\PHI_\beta=
\begin{bmatrix} \Phi_{11} & \Phi_{12} \\ \Phi_{21} & \Phi_{22} \end{bmatrix}
\end{equation}
By \eqref{eq21}, taking into account \eqref{eq27}, one has:
\begin{equation}
\label{eq29}
\PHI_\beta=\begin{bmatrix} R_{\alpha\beta} & -VR_{\alpha\beta} \\ -V_0 R_{\alpha\beta} & VV_0 R_{\alpha\beta}
\end{bmatrix} \cdot\begin{bmatrix} 2VR_{\beta} & VR_{\beta} \\
VR_{\beta} & R_{\beta} \end{bmatrix}\cdot d_\beta^{-1}.
\end{equation}
Hence, taking into account the relation $d_\beta:=(2V-V^2)R_\beta^2$, one obtains
\begin{equation}
\label{eq30}
\begin{split}
&\Phi_{11}=R_{\alpha\beta}R_\beta^{-1},\\
&\Phi_{21}=-V_0R_{\alpha\beta}R_\beta^{-1},\\
&\Phi_{12} = \Phi_{22} =0.
\end{split}\end{equation}
So the  matrix
\begin{equation}
\label{eq31}
\PHI_\beta=\begin{bmatrix} R_{\alpha\beta}R_\beta^{-1} & 0 \\ -V_0R_{\alpha\beta}R_\beta^{-1} & 0
\end{bmatrix},
\end{equation}
which is equivalent to \eqref{eq17}. Theorem \ref{thm3} is proved.
\end{proof}
\vskip10pt

\section{The filtering error}
\vskip10pt

Let's denote the filtering mean square estimation error
\begin{equation}
\label{eq32}
\Gamma(k)=E\biggl({\alpha}(k)-\widehat{\alpha}(k)\biggr)^2+
E\biggl(\Delta{\alpha}(k+1)-\Delta\widehat{\alpha}(k+1)\biggr)^2.
\end{equation}
By stationarity of the processes $\alpha(k)$ and $\beta(k)$, $k\geq0$, we shall skip the parameter $k$ where it is considered possible and convenient.\\
By the normal correlation theorem \cite[Theorem 13.3]{LIP-SH1974}, the mean square error of the filtering is expressed as the trace of the following error matrix:
\begin{equation}
\label{eq31}
\begin{split}
\GAM&=\begin{bmatrix} \Gamma_{11} &  \Gamma_{12} \\  \Gamma_{21} &  \Gamma_{22} \end{bmatrix}=
cov[\alpha(k),\Delta\alpha(k+1)\,|\,\beta(k), \Delta\beta(k+1)] \\ &=\mathbb{R}_{\alpha}-\mathbb{R}_{\alpha\beta}\mathbb{R}_{\beta}^{-1}\mathbb{R}_{\alpha\beta}^*
=\mathbb{R}_{\alpha}[\mathbb{I}-\underbrace{\mathbb{R}_{\alpha}^{-1} \mathbb{R}_{\alpha\beta}}_{\PHI_\alpha}
\underbrace{\mathbb{R}_{\beta}^{-1}\mathbb{R}_{\alpha\beta}^*}_{\PHI_\beta^*}]
\ , \ \ \forall k\geq0.
\end{split}
\end{equation}
Let's denote $\mathcal{F}_k^\beta$ the natural increasing sequence of $\sigma$-algebras of events, generated by the trajectories of the filtering DMD $\beta(k)$, $k\geq0$.
Then the elements of error matrix $\GAM$ are defined as:
\[
\left.
\begin{aligned}
&\Gamma_{11}=E\bigl[\bigl(\alpha(k)-\widehat{\alpha}(k)\bigr)^2\,|\,\mathcal{F}_k^\beta\bigr],\\
&\Gamma_{12}=E\bigl[(\alpha(k)-\widehat{\alpha}(k))(\Delta\alpha(k+1)-
\Delta\widehat{\alpha}(k+1))\,|\,\mathcal{F}_k^\beta\bigr],\\
&\Gamma_{21}=E\bigl[(\Delta\alpha(k+1)-\Delta\widehat{\alpha}(k+1))
(\alpha(k)-\widehat{\alpha}(k))\,|\,\mathcal{F}_k^\beta\bigr],\\
&\Gamma_{22}=E\bigl[\bigl(\Delta\alpha(k+1)-\Delta\widehat{\alpha}(k+1)\bigr)^2\,|\,\mathcal{F}_k^\beta\bigr],
\end{aligned}
\right\}
\ , \ \ \forall k\geq0,
\]
and
\[
\mathbb{R}_{\alpha}=\begin{bmatrix} 1 & -V_0 \\ -V_0 & 2V_0 \end{bmatrix}\cdot R_\alpha^2,
\]
the covariation matrix $\mathbb{R}_{\alpha\beta}$ is defined in \eqref{eq20} and the term $\PHI_\beta$ is defined in formula \eqref{eq31}. \\
Let's calculate the term $\PHI_\alpha$.
\begin{equation}
\label{eq32}
\begin{split}
\PHI_\alpha=\mathbb{R}_{\alpha}^{-1} \mathbb{R}_{\alpha\beta}=
\begin{bmatrix} R_\alpha^\Delta & R_\alpha^0 \\ R_\alpha^0 & R_\alpha \end{bmatrix} \cdot &
\begin{bmatrix} R_{\alpha\beta} & R_{\alpha\beta}^0 \\ R_{\alpha\beta}^0 & R_{\alpha\beta}^\Delta \end{bmatrix} \cdot
d_\alpha^{-1},\\ &d_\alpha:=V_0(2-V_0)R_\alpha^2.
\end{split}\end{equation}
So that
\begin{equation}
\label{eq34}
\PHI_\alpha=\begin{bmatrix}1&-V\\0&0\end{bmatrix}\cdot
R_{\alpha\beta}R_{\alpha}^{-1}.
\end{equation}
Next, using \eqref{eq17}, one obtains
\begin{equation}
\label{eq35}
\PHI_\beta^*=\mathbb{R}_\beta^{-1}\mathbb{R}_{\alpha\beta}^*=
\begin{bmatrix}1&-V_0\\0&0\end{bmatrix}\cdot R_{\alpha\beta}R_{\beta}^{-1}.
\end{equation}
So
\begin{equation}
\label{eq36}
\begin{split}
\PHI_\alpha\cdot\PHI_\beta^*&=\begin{bmatrix}1&-V\\0&0\end{bmatrix}
\cdot\begin{bmatrix}1&-V_0\\0&0\end{bmatrix}
\cdot R_{\alpha\beta}^2R_{\alpha}^{-1}R_{\beta}^{-1}=\\
&=\begin{bmatrix}1&-V_0\\0&0\end{bmatrix}\cdot
R_{\alpha\beta}^2R_{\alpha}^{-1}R_{\beta}^{-1}.
\end{split}\end{equation}
Hence
\begin{equation}
\label{eq37}
\begin{split}
\mathbb{I}-\PHI_\alpha\PHI_\beta^* &=\begin{bmatrix}1&0\\0&1\end{bmatrix}-
\begin{bmatrix}1&-V_0\\0&0\end{bmatrix}\underbrace{R_{\alpha\beta}^2R_\alpha^{-1}R_\beta^{-1}}_
{=:\Gamma_{\alpha\beta}}= \\
& =\begin{bmatrix}1-\Gamma_{\alpha\beta}&V_0\Gamma_{\alpha\beta}\\0&1\end{bmatrix}.
\end{split}
\end{equation}
So the filtering error matrix \eqref{eq31} has the following form:
\begin{equation}
\label{eq38}
\begin{split}
\GAM=\mathbb{R}_\alpha(\mathbb{I}-\PHI_\alpha\PHI_\beta^*) &=\begin{bmatrix}1&-V_0 \\-V_0&2V_0\end{bmatrix}R_\alpha^2
\cdot\begin{bmatrix}1-\Gamma_{\alpha\beta}&V_0\Gamma_{\alpha\beta}\\0&1\end{bmatrix}=\\
&=
\begin{bmatrix}1-\Gamma_{\alpha\beta}&-V_0(1-\Gamma_{\alpha\beta})\\
-V_0(1-\Gamma_{\alpha\beta})&V_0(2-V_0\Gamma_{\alpha\beta})\end{bmatrix}\cdot R_\alpha^2.
\end{split}
\end{equation}
Using the trivial identity $2-V_0\Gamma=2-V_0+V_0(1-\Gamma)$,
one has:
\begin{equation}
\label{eq39}
\begin{split}
\GAM &=R_\alpha^2 \cdot (1-\Gamma_{\alpha\beta})\cdot \begin{bmatrix}1&-V_0\\-V_0&V_0^2\end{bmatrix}+R_\alpha^2\cdot
\begin{bmatrix}0&0\\0&V_0(2-V_0)\end{bmatrix}.
\end{split}
\end{equation}
and considering the stationarity condition $\sigma_0^2=R_\alpha V_0(2-V_0)$, one obtains the following equivalence of \eqref{eq38}:
\begin{equation}
\label{eq40}
\GAM=R_\alpha^2 \cdot (1-\Gamma_{\alpha\beta})\cdot
\begin{bmatrix}1&-V_0\\-V_0&V_0^2\end{bmatrix}+
R_\alpha\cdot\begin{bmatrix}0&0\\0&\sigma_0^2\end{bmatrix}.
\end{equation}
Hence
\begin{align*}
&\Gamma_{11}=R_\alpha^2 \cdot (1-\Gamma_{\alpha\beta}),
&\Gamma_{12}=-V_0R_\alpha^2 \cdot (1-\Gamma_{\alpha\beta}),&\quad \\
&\Gamma_{21}=-V_0R_\alpha^2 \cdot (1-\Gamma_{\alpha\beta}),
&\Gamma_{22}=V_0^2R_\alpha^2 \cdot (1-\Gamma_{\alpha\beta})+
R_\alpha^2\cdot(2V_0-V_0^2).
\end{align*}

\vskip10pt
\section{The filtering empirical estimation}

\vskip10pt
\noindent In real physical observations, the condition of mutual uncorrelatedness \eqref{eq24} is practically not satisfied. Therefore, the covariance characteristics \eqref{eq24} should be taken into account in the covariance analysis of filtering, if they are nonzero. The corresponding correction terms are subject to estimates, based on the filtering equation \eqref{eq16}.
\vskip10pt

\noindent We will explore the best estimate (in the mean square sense), determined by the following empirical filtering equation:
\begin{equation}
\label{eq42}
\begin{pmatrix}\widehat{\alpha}(k), & \Delta\widehat{\alpha}(k+1)\end{pmatrix}= \PHI_\beta^T\cdot\begin{pmatrix} \beta(k) \\ \Delta\beta(k+1)\end{pmatrix} \ , \ \ k\geq0,
\end{equation}
with the empirical filtering matrix
\begin{equation}\label{eq43}
\PHI_\beta^T=\mathbb{R}_{\alpha\beta}^T(\mathbb{R}_{\beta}^T)^{-1},
\end{equation}
where
\begin{equation}
\label{eq44}
\mathbb{R}_{\alpha\beta}^T:=\begin{bmatrix} R_{\alpha\beta}^T & R_{\alpha\beta}^{0T} \\
R_{\beta\alpha}^{0T} & R_{\alpha\beta}^{\Delta T} \end{bmatrix} \ , \ \
\mathbb{R}_{\beta}^T:=\begin{bmatrix} R_{\beta}^T & R_{\beta}^{0T}, \\
R_{\beta}^{0T} & R_{\beta}^{\Delta T} \end{bmatrix}.
\end{equation}
The following  empirical covariances, corresponding to \eqref{eq18} and \eqref{eq23}, are used here:
\begin{equation}
\label{eq45}
\begin{split}
&R_{\alpha\beta}^T:=\frac{1}{T}\sum_{k=0}^{T-1}(\alpha(k)\beta(k)) \ , \ \ R_{\alpha\beta}^{0T}:=\frac{1}{T}\sum_{k=0}^{T-1}(\alpha(k)\Delta\beta(k+1)),\\ &R_{\beta\alpha}^{0T}:=\frac{1}{T}\sum_{k=0}^{T-1}(\beta(k)\Delta\alpha(k+1)) \ , \ \ R_{\alpha\beta}^{\Delta T} :=\frac{1}{T}\sum_{k=0}^{T-1}(\Delta\alpha(k+1)\Delta\beta(k+1)),\\
&R_{\beta}^T:=\frac{1}{T}\sum_{k=0}^{T-1}(\beta(k)^2) \ , \ \ R_{\beta}^{0T}:=\frac{1}{T}\sum_{k=0}^{T-1}(\beta(k)\Delta\beta(k+1)),\\ &R_{\beta}^{\Delta T}:=\frac{1}{T}\sum_{k=0}^{T-1}[(\Delta\beta(k+1))^2].
\end{split}\end{equation}
Note that for one-component correlations one has the representation
\begin{equation}\label{eq46}
R_{\beta}^{0T}=-VR_\beta^T \ , \ \ R_{\beta}^\Delta:=2VR_\beta^T.
\end{equation}
So the factor matrix
\begin{equation}
\label{eq47}
\mathbb{R}_{\beta}^T=\begin{bmatrix} R_{\beta}^T & -VR_{\beta}^T \\
-VR_{\beta}^T & 2VR_{\beta}^T \end{bmatrix}
\end{equation}
has the following inversion:
\begin{equation}
\label{eq48}
(\mathbb{R}_{\beta}^T)^{-1}=\begin{bmatrix} 2VR_{\beta}^T & VR_{\beta}^T \\
VR_{\beta}^T & R_{\beta}^T \end{bmatrix}\cdot (d_\beta^T)^{-1} \ , \ \ d_\beta^T:=(2V-V^2)(R_\beta^T)^2.
\end{equation}
Supposing that the mutual correlations \eqref{eq24} in reality are not null, the empirical covariances are connected by more complex relations, namely
\begin{equation}
\label{eq49}
\begin{split}
&R_{\alpha\beta}^{0T}=-VR_{\alpha\beta}^T+
\sigma\frac{1}{T}\sum_{k=0}^{T-1}(\alpha(k)\Delta W(k+1)),\\
&R_{\beta\alpha}^{0T}=-V_0R_{\alpha\beta}^T+
\sigma_0\frac{1}{T}\sum_{k=0}^{T-1}(\beta(k)\Delta W^0(k+1)),\\
&R_{\alpha\beta}^{\Delta T}=VV_0R_{\alpha\beta}^T-
\sigma\frac{1}{T}\sum_{k=0}^{T-1}(\alpha(k)\Delta W(k+1))-\\
& \hskip1.5cm - \sigma_0\frac{1}{T}\sum_{k=0}^{T-1}(\beta(k)\Delta W^0(k+1))+
\sigma\sigma_0\frac{1}{T}\sum_{k=0}^{T-1}(\Delta W(k+1)\Delta W^0(k+1)).
\end{split}\end{equation}

Taking into account \eqref{eq44} and \eqref{eq47}, one has
\begin{equation}
\label{eq50}
\mathbb{R}_{\alpha\beta}^T=\begin{bmatrix} R_{\alpha\beta}^T & -VR_{\alpha\beta}^{T}+A^T\\ -V_0R_{\alpha\beta}^T+B^T & VV_0R_{\alpha\beta}^T+C^T \end{bmatrix}
\end{equation}
where
\begin{equation}
\label{eq51}
\begin{split}
&A^T:=\sigma\frac{1}{T}\sum_{k=0}^{T-1}(\alpha(k)\Delta W(k+1)),\\
&B^T:=\sigma_0\frac{1}{T}\sum_{k=0}^{T-1} (\beta(k)\Delta W^0(k+1)),\\
&C^T:=-
\sigma\frac{1}{T}\sum_{k=0}^{T-1}(\alpha(k)\Delta W(k+1))
- \sigma_0\frac{1}{T}\sum_{k=0}^{T-1}(\beta(k)\Delta W^0(k+1))+\\
&\hskip5cm+\sigma\sigma_0\frac{1}{T}\sum_{k=0}^{T-1}(\Delta W(k+1)\Delta W^0(k+1)).
\end{split}\end{equation}
Now our task is to express the filtering matrix in the terms of empirical covariances.\\
By the definition \eqref{eq43} one has
\begin{equation}\label{eq52}
\PHI_\beta^T=\begin{bmatrix} R_{\alpha\beta}^T & -VR_{\alpha\beta}^{T}+A^T \\
-V_0R_{\beta\alpha}^{T}+B^T & VV_0R_{\alpha\beta}^T+C^T \end{bmatrix}\cdot
\begin{bmatrix} 2VR_{\beta}^T & VR_{\beta}^T, \\
VR_{\beta}^T & R_{\beta}^T \end{bmatrix}\cdot d_\beta^{-1}.
\end{equation}
Taking into account the relation $d_\beta^T=\mathcal{E}(R_\beta^T)^2=
\sigma\cdot R_\beta^T$, one obtains:
\begin{equation}
\label{eq54}
\begin{split}
& \PHI_{11}^T=R_{\alpha\beta}^T(R_\beta^T)^{-1}
                             +VA^T\cdot\bigl(\sigma R_\beta^T\bigr)^{-1};\\
& \PHI_{12}^T= \ \ \ \ \ \ \ \ \ \ \ \ \ \  A^T\cdot(\mathcal{E}R_\beta^T)^{-1};\\
& \PHI_{21}^T=-V_0R_{\alpha\beta}^T(R_\beta^T)^{-1}
                             +V(2B^T+C^T)\cdot(\sigma R_\beta^T)^{-1};\\
& \PHI_{22}^T= \ \ \ \ \ \ \ \ \ \ \ \ \ \ \ \ \
                             (VB^T+C^T)\cdot(\sigma R_\beta^T)^{-1}.
\end{split}\end{equation}
which one can rewrite in the matrix form as
\begin{equation}\label{eq55}
\PHI_\beta^T=\begin{bmatrix} 1 & 0 \\ -V_0 & 0 \end{bmatrix}
R_{\alpha\beta}^T (R_\beta^T)^{-1}+
\begin{bmatrix} VA^T & A^T, \\
V(2B^T+C^T) & VB^T+C^T \end{bmatrix}
(\sigma R_\beta^T)^{-1}.
\end{equation}
The empirical matrix representation \eqref{eq55} contains two terms. The first addendum defines the filtering matrix under conditions of  uncorrelatedness \eqref{eq24} of the stochastic components of signal and filter. The second addendum defines additional statistical estimates, generated by the correlation of the stochastic components of signal and filter.


\begin{thebibliography}{99}

\bibitem{DK_6}
Koroliouk D. Two component binary statistical experiments with persistent linear regression. - Theor. Probability and Math. Statist., AMS, No. 90, 2015, 103-114.

\bibitem{DK_05}
Koroliouk D., Koroliuk V.S., N.Rosato N. Equilibrium Process in Biomedical Data Analysis: the Wright-Fisher Model. - Cybernetics and System Analysis, Springer NY, 2014, vol. 50, No. 6, 890-897.

\bibitem{DK_8}
Koroliouk D.  Binary statistical experiments with persistent nonlinear regression. - Theor. Probability and Math. Statist., AMS,  No. 91, 2015, 71-80.

\bibitem{DK_19}
Koroliouk D., Koroliuk V.S., Nicolai E., Bisegna P., Stella L., Rosato N. A statistical model of macromolecules dynamics for Fluorescence Correlation Spectroscopy data analysis. - Statistics, Optimization and Information Computing (SOIC) . - Vol. 4, 2016, 233-242.

\bibitem{DK_13}
Koroliouk D. Stationary statistical experiments and the optimal  estimator for a predictable component. - Journal of Mathematical Sciences,  Vol. 214, No. 2, 2016, 220-228.

\bibitem{LIP-SH1974}
Liptser R.Sh., Shiryaev A.N. Statistics of Random Processes. II. Applications. - Springer, Berlin / Heidelberg 2001, 402 p.

\bibitem{MMS_01}
Moklyachuk M., Masyutka O., Sidei M. Minimax Extrapolation of Multidimensional Stationary Processes with Missing Observations. - Intern. Journ. Math. Models and Methods in Appl. Sci., 2018, 12, 94-105.

\bibitem{MMS_02}
Masyutka O., Moklyachuk M., Sidei M. Filtering of Multidimensional Stationary Processes with Missing Observations. - Universal Journal of Mathematics and Applications, 2019, 2(1), 24-32.

\end{thebibliography}
\end{document}